\documentclass[12pt]{amsart}
\usepackage{amssymb, amscd, amsmath, amsthm, latexsym, enumerate}

\renewcommand{\geq}{\geqslant}
\renewcommand{\leq}{\leqslant}

\newtheorem{theorem}{Theorem}
\newtheorem{lemma}[theorem]{Lemma}
\newtheorem{cor}[theorem]{Corollary}

\newtheorem*{thm}{Theorem}

\newtheorem*{cor*}{Corollary}

\begin{document}
\title[4-manifolds with 3-manifold fundamental groups]
{Homotopy types of 4-manifolds with 3-manifold fundamental groups}

\author{Jonathan A. Hillman }
\address{School of Mathematics and Statistics\\
     University of Sydney, NSW 2006\\
      Australia }

\email{jonathanhillman47@gmail.com}

\begin{abstract}
We show that the homotopy type of a 4-manifold $M$ whose fundamental group 
is a finitely presentable $PD_3$-group $\pi$
and with $w_1(M)=w_1(\pi)$ is determined by $\pi$,  $\Pi=\pi_2(M)$,  
$k_1(M)$ and the equivariant intersection pairing $\lambda_M$.
\end{abstract}



\maketitle
The basic algebraic invariants of a closed 4-manifold $M$ are the fundamental group 
$\pi=\pi_1(M)$,  the $\mathbb{Z}[\pi]$-module $\Pi=\pi_2(M)$,
the equivariant homotopy intersection pairing $\lambda_M$ on $\Pi$,
the first $k$-invariant $\kappa=k_1(M)\in{H^3(\pi;\Pi)}$,
the Euler characteristic $\chi(M)$,  
and the Stiefel-Whitney classes $w=w_1(M)$ and $w_2(M)$.
(Strictly speaking, $\lambda_X$ determines $\Pi$,
and it also determines $w$ if $\pi\not=0$.)
These invariants determine the stable homeomorphism type of $M$
(with respect to sums with $S^2\times{S^2}$), 
if $\pi$ is the group of an aspherical closed orientable 3-manifold \cite{KLPT}.
(The $k$-invariant is determined by the other data in this situation.)

We shall show that the homotopy type of a $PD_4$-complex $X$ 
whose fundamental group is a finitely presentable $PD_3$-group $\pi$,
and with  $w_1(X)=w_1(\pi)$ is determined by the  invariants
$\pi$, $\Pi$, $\kappa$ and $\lambda_X$.
(However we do not yet know the possible values of $\Pi$,
$\kappa$ or $\lambda_X$.)

I would like to thank D.Kasprowski for pointing out a slip in the first version of this note,
in my invocation of surgery in Corollary 9 of that version.
(This corollary has since been deleted.)
Since first submitting this note to the arXiv, I have learned that Kasprowski,
Powell and Ray have independently found similar arguments for the main result \cite{KPR}.

\section{Notation and terminology}

Let $\pi$ be a finitely presentable $PD_3$-group,
with orientation character $w=w_1(\pi)$,
and let $Y=K(\pi,1)$. 
We may assume that $Y=Y_o\cup{D^3}$, 
where $(Y_o,S^2)$  is a $PD_3$-pair
and $Y_o$ is cohomologically 2-dimensional.

Let $X$ be a $PD_4$-complex such that 
$\pi=\pi_1(X)\cong\pi$ and $w_1(X)=w_1(\pi)$,
and let $p:\widetilde{X}\to{X}$ be the universal covering.
The homology of $\widetilde{X}$ is given by 
$H_i(\widetilde{X};\mathbb{Z})=H_i(X;\mathbb{Z}[\pi])$, for all $i$.
We assume that $\pi$ acts on the left of $\widetilde{X}$,
and so these are left $\mathbb{Z}[\pi]$-modules.
Since $\pi$ has one end, $H_i(X;\mathbb{Z}[\pi])=0$ if $i\not=0,2$.
Since $c.d.\pi=3$, $X$ is not aspherical, 
and so $\Pi={H_2(X;\mathbb{Z}[\pi])}\not=0$.

The homologies of $\widetilde{X}$ and $X$ 
are related by the Cartan-Leray spectral sequence for the covering, 
which has the form
\[
E_{p,q}^2=H_p(\pi;H_q(X;\mathbb{Z}[\pi]))\Rightarrow{H_{p+q}(X;\mathbb{Z})}.
\]
(Note that the groups $H_*(\pi;A)=Tor_*^{\mathbb{Z}[\pi]}(\mathbb{Z},A)$
with coefficients in a left module $A$ are defined via a resolution of $\mathbb{Z}$ by 
{\it right\/} $\mathbb{Z}[\pi]$-modules.)
There is also a Universal Coefficient spectral sequence
\[
E_2^{p,q}=Ext^q_{\mathbb{Z}[\pi]}(H_p(X;\mathbb{Z}[\pi]), \mathbb{Z}[\pi])
\Rightarrow{H^{p+q}(X;\mathbb{Z}[\pi])},
\]
which relates the homology and equivariant cohomology of $\widetilde{X}$.
 
Let $\varepsilon:\mathbb{Z}[\pi]\to\mathbb{Z}$ be the augmentation homomorphism,
with kernel the augmentation ideal $I_\pi$,
and let $j_\pi:I_\pi\to\mathbb{Z}[\pi]$ be the natural inclusion.
Applying the functor $Hom_{\mathbb{Z}[\pi]}(\mathbb{Z},-)$ 
to the exact sequence
\begin{equation*}
\begin{CD}
0\to{I_\pi}\to\mathbb{Z}[\pi]@>\varepsilon>>\mathbb{Z}\to0
\end{CD}
\end{equation*}
gives $H^0(\pi;I_\pi)=0$ and
$H^i(\pi;\mathbb{Z})\cong{H^{i+1}(\pi;I_\pi)}~\mathrm{for}~i=0$ or 1,
while applying the functor $Hom_{\mathbb{Z}[\pi]}(-,\mathbb{Z}[\pi])$ gives
$Hom_{\mathbb{Z}[\pi]}(I_\pi,\mathbb{Z}[\pi])\cong\mathbb{Z}[\pi]$ 
(with generator $j_\pi$) and
$Ext^i_{\mathbb{Z}[\pi]}(I_\pi,\mathbb{Z}[\pi])\cong{H^{i+1}(\pi;\mathbb{Z}[\pi])},
~\mathrm{for}~ i>0$ \cite[Lemma 3.2]{Pl86}.
Note also that $Hom_{\mathbb{Z}[\pi]}(I_\pi,I_\pi)\cong\mathbb{Z}[\pi]$,
generated by $id_{I_\pi}$ \cite[Lemma 3.3]{Pl86}.
(The latter result is not stated explicitly in \cite[Lemma 3.3]{Pl86},
which is formulated so as to apply also to the case when $\pi$ is finite.)

If $M$ is a left $\mathbb{Z}[\pi]$-module let $M^w=\mathbb{Z}^w\otimes{M}$ 
be the left $\mathbb{Z}[\pi]$-module with the same underlying abelian group 
and diagonal left $\pi$-action,
given by $g(1\times{x})=w(g)(1\times{gx})$ for all $g\in\pi$ and $x\in\Pi$.
Then $\mathbb{Z}[\pi]^w\cong\mathbb{Z}[\pi]$, 
since we may define an isomorphism  
$f:\mathbb{Z}[\pi]\to\mathbb{Z}[\pi]^w$ by $f(g)=w(g)\otimes{g}$ for all $g\in\pi$.
The linear extension of $w$ defines the {\it $w$-twisted augmentation\/} 
$\varepsilon_w:\mathbb{Z}[\pi]\to\mathbb{Z}^w$,
with kernel $I_\pi^w$.
Arguments similar to those of the previous paragraph apply to $I_\pi^w$.

If $R$ is a right $\mathbb{Z}[\pi]$-module let $\overline{R}$ be the left module
with the same underlying group and $\mathbb{Z}[\pi]$-action determined 
by $g.r=w(g)rg$ for all $r\in{R}$  and $g\in\pi$.
We use a similar strategy and notation to obtain a right module $\overline{L}$ 
from a left $\mathbb{Z}[\pi]$-module $L$.
Free right modules give rise to free left modules of the same rank, and conversely.
We may define the dual of a left module $M$ as the left module
$M^\dagger=\overline{Hom_{\mathbb{Z}[\pi]}(M,\mathbb{Z}[\pi])}$.

Two (left) $\mathbb{Z}[\pi]$-modules $L$ and $L'$ are {\it stably projective equivalent\/} 
if $L\oplus{P}\cong{L'}\oplus{P'}$ for some finitely generated projective 
$\mathbb{Z}[\pi]$modules $P,P'$.
They are {\it stably equivalent\/} if we may assume that $P$ and $P'$ are
each free modules.
We shall let $[L]_{pr}$ and $[L]$ denote the equivalence classes 
corresponding to these two equivalence relations.
As our concern in this paper is mainly with finite $PD_4$-complexes, 
which correspond most closely to manifolds, 
stable equivalence is the more useful notion.
However our arguments apply with little change to the more general setting 
of finitely dominated $PD_4$-complexes
(and even $PD_4$-spaces in the sense of \cite{PDD3}),
for which the broader notion of stably projective equivalence is needed.

If $A$ is an abelian group let $rk(A)=\dim_\mathbb{Q}\mathbb{Q}\otimes{A}$ be its rank.

\section{the basic examples}

Let $\tau$ be a self-homeomorphism of $S^2\times{S^1}$
which does not extend over $S^2\times{D^2}$.
(There is an unique isotopy class of such maps $\tau$.)
Then $X(\pi)=Y_o\times{S^1}\cup{S^2}\times{D^2}$
and  $X(\pi)_\tau=Y_o\times{S^1}\cup_\tau{S^2}\times{D^2}$
are $PD_4$-complexes with fundamental group $\pi$,
orientation character $w=w_1(Y)$ and  Euler characteristic 2.
Since $X(\pi)=\partial(Y_o\times{D^2})$, it retracts onto $Y_o$.

The arguments of \cite[\S2]{Pl86} for the case when $Y$ is a 3-manifold 
are essentially homological and apply equally well in our situation.
Let  $U=Y_o\times{S^1}$.
The long exact sequence of the pair $(X(\pi),U)$  
with coefficients $\mathbb{Z}[\pi]$ gives a five-term exact sequence
\[
H_3(X(\pi),U;\mathbb{Z}[\pi])\to{H_2(U;\mathbb{Z}[\pi])}\to{H_2(X(\pi);\mathbb{Z}[\pi])}\to
\]
\[{H_2(X(\pi),U;\mathbb{Z}[\pi])}\to{H_1(U;\mathbb{Z}[\pi])}\to0,
\]
since $H_i(X;\mathbb{Z}[\pi])=0$ for $i\not=0, 2$.
Now $H_3(X(\pi),U;\mathbb{Z}[\pi])=0$ and 
$H_2(X(\pi),U;\mathbb{Z}[\pi])\cong\mathbb{Z}[\pi]$, 
by excision, 
while\\ $H_2(U;\mathbb{Z}[\pi])\cong{H_2(Y_o;\mathbb{Z}[\pi])}(\cong
\pi_2(Y_o))\cong\mathbb{Z}[\pi]$ and $H_1(U;\mathbb{Z}[\pi])\cong\mathbb{Z}$.
Hence this sequence reduces to
\[
0\to\pi_2(Y_o)\cong\mathbb{Z}[\pi]\to\pi_2(X(\pi))\to\mathbb{Z}[\pi]\to\mathbb{Z}\to0.
\]
Hence $\pi_2(X(\pi))$ is an extension of $I_\pi$ by $\pi_2(Y_o)\cong\mathbb{Z}[\pi]$.
The extension splits,
since  $Ext^1_{\mathbb{Z}[\pi]}(I_\pi,\mathbb{Z}[\pi])=H^2(\pi;\mathbb{Z}[\pi])=0$,
and so  $\pi_2(X(\pi))\cong{\mathbb{Z}[\pi]\oplus{I_\pi}}$.
The retraction of $X(\pi)$ onto $Y_o$ determines a splitting.

A similar argument shows that 
$\pi_2(X(\pi)_\tau)\cong{\mathbb{Z}[\pi]\oplus{I_\pi}}$ also,
but we do not know whether $X(\pi)_\tau$ retracts onto $Y_o$.

If $\pi$ has a balanced presentation then 
there is a finite 2-complex $K$ with $\pi_1(K)\cong\pi$ and $\chi(K)=1$.
Let $N$ be a 4-dimensional handlebody thickening of $K$.
Then the double of $N$ is a closed 4-manifold $M$ with 
$\pi_1(M)\cong\pi$ and $\chi(M)=2$, and which retracts onto $K$.

If $Y$ is a closed 3-manifold then the corresponding closed 4-manifolds
are the manifolds obtained by elementary surgery on the second factor
of $Y\times{S^1}$. 
(There are two possible framings of the normal bundle.)
Plotnick uses manifold topology, 
first to define a splitting of the above exact sequence and then 
to show that elements in the image of $\pi_2(Y_o)$ in $\pi_2(X)$ 
have self-intersection 0, in either case, 
and that $\pi_2(X(\pi))$ is the direct sum of two summands which are
self-annihilating with respect to $\lambda_X$.
The equivariant homotopy intersection pairings of these 4-manifolds 
are not isometric \cite[Theorem 3.1]{Pl86}.

\section{$\Pi$, $\chi$ and $\lambda_X$}

In this section we shall summarize the key properties of $\Pi$ and $\chi$,
which were determined in \cite[Theorem 3.13]{FMGK}, 
and define the equivariant intersection pairing $\lambda_X$,
using the cohomological formulation.

We shall first state without proof a result from \cite{FMGK}.
\begin{thm}
\cite[Theorem 3.13]{FMGK}
Let $X$ be a $PD_4$-complex such that $\pi=\pi_1(X)$ is a finitely presentable
$PD_3$-group and $w_1(X)=w_1(\pi)$.
Then $\chi(X)\geq2$, $[\Pi]_{pr}=[I_\pi]_{pr}$ and $\Pi^\dagger$ is projective.
\end{thm}

The Euler characteristic is in fact determined by $\pi$ and $\Pi$.
This follows easily from the invariance of $\chi$ between the pages of a spectral sequence, 
and the special nature of $\Pi$.

\begin{cor}
 $\chi(X)=rk(\mathbb{Z}\otimes_{\mathbb{Z}[\pi]}\Pi)+1-\beta_1(\pi)$.
\end{cor}

\begin{proof}
The only nonzero entries in the Cartan-Leray homology spectral sequence for the universal cover of $X$ are when $0\leq{p}\leq3$ and $q=0$ or 2,
and then $E_{p,0}^2=H_p(\pi;\mathbb{Z})$,
while $E_{0,2}^2=\mathbb{Z}\otimes_{\mathbb{Z}[\pi]}\Pi$ and
$E_{p,2}^2=H_{p+1}(\pi;\mathbb{Z})$ for $p>0$, since $[\Pi]_{pr}=[I_\pi]_{pr}$.
Since $\chi(X)=\Sigma_{p,q}(-1)^{p+q}rk(E_{p,q}^2)$, the corollary follows.
\end{proof}

If $X$ is a closed 4-manifold, $\pi=\pi_1(X)$ and $w=w_1(X)$ 
then geometric intersection numbers can be used to define a
$w$-hermitean equivariant intersection pairing on $\Pi$, 
with values in $\mathbb{Z}[\pi]$.
In the Poincar\'e duality complex case we cannot count geometric intersection numbers
and so we shall use the cohomological formulation of the intersection pairing 
\cite[Proposition 4.58]{Ra} instead.
This formulation is well suited to the application of
\cite[Theorem 2]{Hi20} in Theorem \ref{mainthm} below.

\begin{lemma}
\label{evalseq}
There is an exact sequence of left $\mathbb{Z}[\pi]$-modules
\begin{equation*}
\begin{CD}
0\to\overline{H^2(X;\mathbb{Z}[\pi])}@>{ev}>>\Pi^\dagger\to\mathbb{Z}\to0.
\end{CD}
\end{equation*}
 \end{lemma}
 
\begin{proof}
This follows from the Universal Coefficient spectral sequence,
since $H^2(\pi;\mathbb{Z}[\pi])=0$, 
$H^3(\pi;\mathbb{Z}[\pi])\cong\mathbb{Z}$, 
and $H^3(X;\mathbb{Z}[\pi])=0$.
 (See  \cite[Lemma 3.3]{FMGK}.)
 \end{proof}
 
Let $D:\Pi\to{H^2(X;\mathbb{Z}[\pi])}$ be the isomorphism given by Poincar\'e duality.
Then the intersection pairing may be defined by
\[
\lambda_X(u,v)=ev(v)([X]\cap{u}),\quad\forall~u,v\in{H^2(X;\mathbb{Z}[\pi])}.
\]
(See \cite[Proposition 4.58]{Ra}.)
Then $\lambda_X(gu,gv)=w(g)\lambda_X(u,v)$ for all $g\in\pi$ 
and $u,v\in{H^2(X;\mathbb{Z}[\pi])}$.
Since $\Pi\not=0$ and $ev$ is a monomorphism, $\lambda_X$ is non-zero,
and so $w$ is determined by $\lambda_X$.

It is clear from the argument in \cite[Theorem 3.13]{FMGK}
that if $X$ is finite and $\pi$ is of type $FF$ then  
$\Pi\oplus\mathbb{Z}[\pi]^r\cong
\mathbb{Z}[\pi]^{\chi(X)-1}\oplus{I_\pi}\oplus\mathbb{Z}[\pi]^r$ 
for $r$ large, and so $[\Pi]=[I_\pi]$.
(In fact any two of the conditions ``$X$ is finite", 
``$\pi$ is of type $FF$" and ``$[\Pi]=[I_\pi]$" imply the third.)
The minimal value $\chi(X)=2$ is realized by the complexes $X(\pi)$ and $X(\pi)_\tau$ 
defined above.

In the 3-manifold group case $\widetilde{K}_0(\mathbb{Z}[\pi])=0$,
by work of Farrell and Jones, 
anticipating the Geometrization Theorem \cite{FJ87}.
In this case $X$ is finite and $\pi$ is of type $FF$,
and we again have $[\Pi]=[I_\pi]$.
This allows some of our statements to be simplified.
It is widely expected that $\widetilde{K}_0(\mathbb{Z}[G])=0$ for any
torsion-free group $G$, and we may assume this whenever convenient.
However even with this assumption there may be difficulties.
If $\pi$ is polycyclic but not abelian then there are ideals 
$J<\mathbb{Z}[\pi]$ such that 
$\mathbb{Z}[\pi]\oplus{J}\cong\mathbb{Z}[\pi]^2$ but which are not free \cite{Ar81}.

The case $\pi=\mathbb{Z}^3$ is exceptional,
for then all projective $\mathbb{Z}[\pi]$-modules are free.
Hence $\Pi^\dagger$ is free,
and it follows from Lemma \ref{evalseq}
that $\Pi\cong\mathbb{Z}[\pi]^{\chi(X)-1}\oplus{I_\pi}$.
Moreover,  if $K$ is any finite 2-complex with $\pi_1(K)\cong\mathbb{Z}^3$ and
$\chi(K)=1$ then $\pi_2(K)$ is free of rank 1, and so $K\simeq{T^3_o}$.

\section{the main theorem}

Let $\Gamma$ be the quadratic functor of Whitehead.
Let $L$ be a finitely generated left $\mathbb{Z}[\pi]$-module,
and let $Her_w(L^\dagger)$ be the abelian group of $w$-hermitean pairings on $L^\dagger$.
Then there is a natural homomorphism 
$B_L: \mathbb{Z}^w\otimes_{\mathbb{Z}[\pi]}\Gamma(L)\to{Her_w(L^\dagger)}$,
which is an isomorphism if $\pi$ is 2-torsion-free and $L$ is projective 
\cite[Theorem 1]{Hi20}.

If $\pi$ is a $PD_3$-group then it is torsion-free.
However the modules of interest to us are not projective, 
since $I_\pi$ has projective dimension 2.
We shall show that when $\Pi$ is stably equivalent to $I_\pi$ then $B_\Pi$ remains injective.
We first recall some details about  $\Gamma$ from \cite[Chapter 1.\S4]{Ba}.
Let $\gamma_A:A\to\Gamma(A)$ be the canonical quadratic map.
We may define a homomorphism $[-]:A\odot{A}\to\Gamma(A)$ by 
\[
[a\odot{b}]=\gamma(a+b)-\gamma(a)-\gamma(b).
\]
Then $[a\odot{a}]=2\gamma(a)$ for all $a\in{A}$.

As abelian groups, 
$\Gamma(A\oplus{B})\cong\Gamma(A)\oplus\Gamma(B)\oplus(A\otimes{B})$.
If $A$ and $B$ are $\mathbb{Z}[G]$-modules the summands are invariant 
under the action of $G$,
and so this direct sum splitting is a $\mathbb{Z}[G]$-module splitting.

The following lemma is close to the first part of \cite[Lemma 2.3]{HK88}
(which considered only finite groups $G$).

\begin{lemma}
\label{HKrerun}
Let $G$ be a group.
Then $\Gamma(\mathbb{Z}[G])\cong\mathbb{Z}[G]\oplus\Gamma(I_G)$
as a left $\mathbb{Z}[G]$-module.
\end{lemma}

\begin{proof}
Let $i_g=g-1$, for $g\in{G}$, 
and let $j:\mathbb{Z}\to\mathbb{Z}[G]$ be the canonical ring homomorphism.
Then $I_G$ is free with basis $\{i_g\mid{g\in{G}}\}$, and
$\mathbb{Z}[G]\cong\mathrm{Im}(j)\oplus{I_G}$ as abelian groups.
Hence $\Gamma(\mathbb{Z}[G])$ splits as a direct sum of abelian groups
$\Gamma(\mathbb{Z})\oplus\Gamma(I_G)\oplus(I_G\otimes\mathbb{Z})$.
The middle summand is a $\mathbb{Z}[G]$-submodule,
but the others are not.

The complement of $\Gamma(I_G)$ in $\Gamma(\mathbb{Z}[G])$
is freely generated (as an abelian group)
by the elements $e=\gamma_{\mathbb{Z}[G]}(1)$ and
$\{i_g\otimes1\mid{g\in{G}}\}$,
and so the quotient $\Gamma(\mathbb{Z}[G])/ \Gamma(I_G)$
is freely generated by the images of these elements.
The group $G$ acts on the basis elements by $h.1=1+i_h$
and $h.i_g=i_{hg}-i_h$.
Hence 
\[
2(h.e)=
h.2\gamma_{\mathbb{Z}[G]}(1)=h([1\odot1])\equiv[1\odot1]+2i_h\otimes1
~mod~\Gamma(I_G).
\]
Since $[1\odot1]=2e$ and $\Gamma(\mathbb{Z}[G])/\Gamma(I_G)$ is torsion-free (as an abelian group), 
\[
h.e\equiv{e+i_h\otimes1}~mod~\Gamma(I_G).
\]
We also have
\[
h.(i_g\otimes1)\equiv{i_{hg}\otimes1-i_g\otimes1}~mod~\Gamma(I_G).
\]
Thus the bijection sending $i_g$ to $i_g\otimes1$ and $1$ to $e$ defines 
an isomorphism $\mathbb{Z}[G]\cong\Gamma(\mathbb{Z}[G])/ \Gamma(I_G)$.
Hence $\Gamma(\mathbb{Z}[G])\cong\mathbb{Z}[G]\oplus\Gamma(I_G)$.
\end{proof}

We may strengthen this result as follows.

\begin{lemma}
\label{summand}
If $\Pi\oplus\mathbb{Z}[\pi]^r\cong\mathbb{Z}[\pi]^s\oplus{I_\pi}$ 
then $\Gamma(\Pi)$ is a direct summand of $\Gamma(\mathbb{Z}[\pi]^{s+1})$.
\end{lemma}

\begin{proof}
Since $\Gamma(\Pi)$ is a direct summand of $\Gamma(\Pi\oplus\mathbb{Z}[\pi]^r)$,
it shall suffice to assume that $\Pi\cong\mathbb{Z}[\pi]^s\oplus{I_\pi}$.
We may compare the splittings
\[
\Gamma(\mathbb{Z}[\pi]^s\oplus{I_\pi})=\Gamma(\mathbb{Z}[\pi]^s)\oplus
\Gamma(I_\pi)\oplus(\mathbb{Z}[\pi]^s\otimes{I_\pi})
\]
and
\[
\Gamma(\mathbb{Z}[\pi]^{s+1})=\Gamma(\mathbb{Z}[\pi]^s)\oplus
\Gamma(\mathbb{Z}[\pi])\oplus(\mathbb{Z}[\pi]^s\otimes\mathbb{Z}[\pi]).
\]
If the abelian group underlying a $\mathbb{Z}[\pi]$-module $M$ is free abelian 
with basis $\{m_i\}$ then the tensor products $M\otimes\mathbb{Z}[\pi]$ and 
$\mathbb{Z}[\pi]\otimes{M}$ with the diagonal left $\mathbb{Z}[\pi]$-structures 
are free $\mathbb{Z}[\pi]$-modules with bases $\{m_i\otimes1\}$ and $\{1\otimes{m_i}\}$,
respectively.
Hence
\[
(\mathbb{Z}[\pi]^s\otimes\mathbb{Z}[\pi])\oplus\Gamma(\mathbb{Z}[\pi])\cong
(\mathbb{Z}[\pi]^s\otimes{I_\pi})\oplus\mathbb{Z}[\pi]\oplus\Gamma(I_\pi)\oplus\mathbb{Z}[\pi],
\]
and so $\Gamma(\mathbb{Z}[\pi]^s\oplus{I_\pi})$ is a direct summand of 
$\Gamma(\mathbb{Z}[\pi]^{s+1})$.
\end{proof}
 
If $[\Pi]=[I]$  then $[\Pi^w]=[I^w]$.
Therefore $H_i(\pi;\Pi^w)\cong{H_{i+1}(\pi;\mathbb{Z}^w)}$ for $i>0$.
Hence $H_2(\pi;\Pi^w)\cong\mathbb{Z}$ and $H_3(\pi;\Pi^w)=0$.

\begin{theorem}
\label{mainthm}
Let $X$ be a $PD_4$-complex such that $\pi=\pi_1(X)$ is a $PD_3$-group
and $w_1(X)=w_1(\pi)$.
Then the homotopy type of $X$ is determined by $\pi$, 
$\Pi=\pi_2(X)$,  $\kappa=k_1(X)$ and $\lambda_X$.
\end{theorem}

\begin{proof}
Let  $w=w_1(X)$ and let $[X]_P\in{H_4(P_2(X);\mathbb{Z}^w)}$ be the image of a fundamental class for $X$.
Then the homotopy type of $X$ is determined by 
$P_2(X)$ and $[X]_P$  \cite[Theorem 3.1]{BB}.
(This was first proven in \cite[Theorem 1.1]{HK88}, 
assuming also that $\beta_2(X;\mathbb{Q})>0$.)
The invariants $\pi,\Pi$ and $\kappa$ determine $P_2(X)$.
We shall show that $\lambda_X$ determines $[X]_P$. 

The universal cover of $P_2(X)$ is a $K(\Pi,2)$-space,
and the `boundary" homomorphism $b:H_4(\Pi,2;\mathbb{Z})\cong\Gamma(\Pi)$
of Whitehead is an isomorphism, since $\pi_i(K(\Pi,2))=0$ for $i\not=2$ 
(see \cite[1.3.7]{Ba}.
Hence $\psi=\mathbb{Z}^w\otimes_\Gamma{b}$ is also an isomorphism.
The Cartan-Leray spectral sequences for the universal covers give epimorphisms
$\delta_X:H_4(X;\mathbb{Z}^w)\to{H_2(\pi;\Pi^w)}$ and 
$\delta_P:H_4(P_2(X);\mathbb{Z}^w)\to{H_2(\pi;\Pi^w)}$,
since $c.d.\pi=3$.
Since $\pi$ has one end, $\delta_X$ is an isomorphism,
and so $H_2(\pi;\Pi^w)\cong\mathbb{Z}$.
There is also an exact sequence 
\begin{equation*}
\begin{CD}
0\to\mathbb{Z}^w\otimes_{\mathbb{Z}[\pi]}H_4(\Pi,2;\mathbb{Z})@>\phi>>
{H_4(P_2(X);\mathbb{Z}^w)}@>\delta_P>>{H_2(\pi;\Pi^w)}\to0.
\end{CD}
\end{equation*}
Since $\mathbb{Z}^w\otimes_{\mathbb{Z}[\pi]}\Gamma(\Pi)$ is a direct summand 
of $\mathbb{Z}^w\otimes_{\mathbb{Z}[\pi]}\Gamma(\mathbb{Z}[\pi]^{s+1})$,
by Lemma \ref{summand}, 
and since $B_M$ is an isomorphism if $M$ is a finitely generated projective module \cite[Theorem 2]{Hi20}, $B_\Pi$ is a monomorphism.

Let $\theta:H_4(P_2(X);\mathbb{Z}^w)\to{Her_w(\Pi^\dagger)}$
be the function defined by
\[
\theta(\xi)(u,v)=v(u\cap\xi)\quad\forall~u,v\in{H^2(X;\mathbb{Z}[\pi])}~\mathrm{and}~
\xi\in{H_4(P_2(X);\mathbb{Z}^w)}.
\]
Then $\theta([X]_P)=\lambda_X$,
and $\theta\phi=B_\Pi\psi$.
Hence $\theta\phi$ is  a monomorphism.

Suppose that $X_1$ is a second such $PD_4$-complex 
and $h:P_2(X_1)\to{P_2(X)}$ is a homotopy equivalence 
which induces an isometry $\lambda_{X_1}\cong\lambda_X$.
Then $\theta(h_*[X_1]_P)=\theta([X]_P)$.
Since these pairings are non-trivial,
 the images of $h_*[X_1]_P$ and $[X]_P$ in $H_2(\pi;\Pi^w)$ agree,
and so $h_*[X_1]_P-[X]_P$ is in the image of $\phi$.
Hence $h_*[X_1]_P=[X]_P$, since $\theta\phi$ is a monomorphism,
and so $X_1\simeq{X}$ \cite[Theorem 3.1]{BB}.
\end{proof}

To what extent can we extend this argument to more general 3-manifold groups?
If $\pi\cong{F(r)}$ is a free group then $\Pi$ is free as a $\mathbb{Z}[\pi]$-module and the inclusion $\mathbb{Z}^w\otimes_{\mathbb{Z}[\pi]}H_4(\Pi,2;\mathbb{Z})$
is an isomorphism,
and the conclusion of Theorem \ref{mainthm} holds \cite{Hi04}.
If $\pi$ is a 3-manifold group which is torsion-free but not free then it is a free product
$F(r)*(*_{i=1}^sG_i)$, where the factors $G_i$ are $PD_3$-groups and $s>0$.
In this case the end module $H^1(\pi;\mathbb{Z}[\pi])$ is a free $\mathbb{Z}[\pi]$-module
\cite[Lemma 2]{Hi20}.
Hence $H_3(X;\mathbb{Z}[\pi])$ is free and $\delta_X$ is again an isomorphism.
Using arguments similar to those in \cite[Theorem 3.13]{FMGK},
we may show that $\Pi$ has a finite projective resolution of length 2, 
$\Pi^\dagger$ is a free $\mathbb{Z}[\pi]$-module and 
$Ext_{\mathbb{Z}[\pi]}^1(\Pi,\mathbb{Z}[\pi])=0$.
However in this case it is not known whether $B_\Pi$ is a monomorphism.

Torsion in $\pi$ complicates matters further. 
However the cases when $X$ is orientable and
$\pi\cong\mathbb{Z}\oplus(\mathbb{Z}/2\mathbb{Z})$ or is a free product of cyclic groups
are treated successfully in \cite{KPR}.

\section{$k_1$ and retractions onto $Y_o$}

The first $k$-invariant is an element of $H^3(\pi;\Pi)$,
and is well-defined up to the actions of $Aut(\pi)$ and $Aut_\pi(\Pi)$.
If $Z$ is a cell complex we may assume that the Postnikov 2-stage $f_2(Z):Z\to{P_2(Z)}$ 
is an inclusion, and that $P_2(Z)$ is obtained from $Z$ by adding cells of dimension $\geq4$.

\begin{lemma}
\label{pdnk}
Let $G$ be a $PD_n$-group and let $C_*$ be a projective resolution
of the augmentation module $\mathbb{Z}$ of length $n$ 
such that $C_n\cong\mathbb{Z}[G]$.
Then the class $[C_*]$ of $C_*$ in 
$H^n(G;C_n)=Ext^n_{\mathbb{Z}[G]}(\mathbb{Z},C_n)\cong\mathbb{Z}$ is a generator.
\end{lemma}

\begin{proof}
Let $D_*$ be the chain complex with $D_i=C_i$ for $i<n$
and $D_i=0$ for $i\geq{n}$,
and let $E_*$ be the chain complex with $E_i=D_i$ for $i\not=n-1$ and 
$E_{n-1}=C_{n-1}/Z_{n-1}\oplus{C_n}$.
Then $D_*$ and $E_*$ are of type $(\mathbb{Z},0, C_n, n-1)$ 
as defined in \cite[Definition 7.1]{Do60}.
The $k$-invariant of $D_*$ is represented by the class  $[C_*]$, 
while the $k$-invariant of $E_*$ is 0.

If $[C_*]=0$ then there is a c.h.e.  $f:D_*\to{E_*}$ \cite[Satz 7.6]{Do60},
and since $H_{n-1}(f)$ is an isomorphism we see that $C_*$ is chain homotopy 
equivalent to a sequence in which $\partial_n$ is a split injection.
Hence $H^n(C_*;\mathbb{Z}[G])=0$, 
contrary to hypothesis.

Similarly, if $[C_*]$ is a $p$-fold multiple of some other class for some prime $p$
then $H^n(C_*;\mathbb{F}_p[G])=0$,  again contrary to hypothesis.
Therefore $[C_*]$ is indivisible, and so is a generator of 
$H^n(G;\mathbb{Z}[G])\cong\mathbb{Z}$.
\end{proof}

\begin{lemma}
\label{cc3}
Let $\pi$ be a $PD_3$-group and $z$ be a generator of $H^3(\pi;\mathbb{Z}[\pi])$.
If $\mathcal{M}$ is a finitely generated $\mathbb{Z}[\pi]$-module and
$\omega\in{H^3(\pi;\mathcal{M})}$ then there is a homomorphism 
$h:\mathbb{Z}[\pi]\to\mathcal{M}$
such that $H^3(h;\mathbb{Z}[\pi])(z)=\omega$.
\end{lemma}

\begin{proof}
Let $f:\mathbb{Z}[\pi]^g\to\mathcal{M}$ be an epimorphism.
Since $c.d.\pi=3$ the change of coefficients homomorphism 
$H^3(f):H^3(\pi;\mathbb{Z}^g)\to{H^3(\pi;\mathcal{M})}$ is also an epimorphism.
It is easy to see that every element of $H^3(\pi;\mathbb{Z}^g)$
($\cong\mathbb{Z}^g$)
is the image of $z$ under a change of coefficients homomorphism.
The result follows. 
\end{proof}

In the next theorem we shall assume that $X$ and $Y$ are as defined in \S1 above.

\begin{theorem}
\label{retract}
The Postnikov $2$-stage $P_2(X)$ retracts onto $P_2(Y_o)$ if and only if 
$\Pi\cong\mathbb{Z}[\pi]\oplus{L}$,  for some $L$,
and the image of $k_1(X)$ in $H^3(\pi;\mathbb{Z}[\pi])$ is a generator.
If $X'$ is another such $PD_4$-complex and there
is an isomorphism $f:\pi_1(X')\to\pi_1(X)=\pi$
such that $w_1(X')=f^*w_1(X)$ and that $\pi_2(X')\cong\pi_2(X)=\Pi$ then $P_2(X')\simeq{P_2(X)}$.
\end{theorem}

\begin{proof}
If $P_2(X)$ retracts onto $P_2(Y_o)$ then there is a pair of maps
 $j:P_2(Y_o)\to{P_2(X)}$ and $r:P_2(X)\to{P_2(Y_o)}$ such that 
$rj\sim{id_{Y_o}}$.
It follows immediately that $\pi_2(Y_o)\cong\mathbb{Z}[\pi]$ 
is a direct summand of $\Pi$, and that $j^*k_1(X)=k_1(Y_o)$,
up to the action of automorphisms.
The chain complex $C_*(Y_o;\mathbb{Z}[\pi])$ is chain homotopy equivalent to
a finite projective complex $D_*$ with $D_i=0$ for $i>2$.
The complex $C_*$ with $C_i=D_i$ for $i\not=3$ and $C_3=H_2(D_*)\cong\pi_2(Y_o)$
is a projective resolution of $\mathbb{Z}$, 
and $k_1(Y_o)$ is the class of $C_*$ in $H^3(\pi;\mathbb{Z}[\pi])$.
Hence $k_1(Y_o)$ 
is a generator of $H^3(\pi;\mathbb{Z}[\pi])$,
by Lemma \ref{pdnk}.

Conversely, if the conditions hold then there are morphisms between the algebraic 2-types
$[\pi,\Pi,k_1(X)]$ and $[\pi,\pi_2(Y_o),k_1(Y_o)]$ which can be realized by maps defining
a retraction.

Suppose now that $X'$ is another such $PD_4$-complex realizing $\pi$ and $\Pi$.
We may assume that $\Pi\cong\mathbb{Z}[\pi]\oplus{L}$,
for some $\mathbb{Z}[\pi]$-module $L$,  by Theorem \ref{retract}.
Let $e$ generate a free summand of $\Pi$.
If $k$ and $k'$ in $H^3(\pi;\Pi)$ each project to generators of
$H^3(\pi;\mathbb{Z}[\pi])\cong\mathbb{Z}$ then we may use Lemma \ref{cc3} to
find an automorphism $\phi$ of $\Pi$ such that $\phi(e)=\pm{e}+\ell$, 
for some $\ell\in{L}$, and $\phi|_L=id_L$,
and such that the induced automorphism of $H^3(\pi;\Pi)$ carries $k$ to $k'$.
\end{proof}

This result extends partially an observation in \cite{KLPT}, 
namely that the $k$-invariant plays no role in their stable classification.

\begin{cor}
\label{chi2cor}
$P_2(X)\simeq{P_2(\partial(Y_o\times{D^2}))}\Leftrightarrow
\Pi\cong\mathbb{Z}[\pi]\oplus{I_\pi}$.
\qed
\end{cor}

\section{a special case}

If there is a retraction $r:X\to{Y_o}$ then $r^*H^2(Y_o;\mathbb{Z}[\pi])\cong\mathbb{Z}[\pi]$
is a free direct summand of $\Pi$ which 
is self-annihilating with respect to $\lambda_X$, since $r_*[X]=0$,
and the image of $k_1(X)$ generates the corresponding summand of $H^3(\pi;\Pi)$,
by Theorem \ref{retract}.

In this section we shall show that for each choice of $\pi$ there are at most 
two homotopy types of $PD_4$-complexes $X$ with $\pi_1(X)\cong\pi$,
$w_1(X)=w_1(\pi)$ and $\chi(X)=2$,
and which satisfy the conditions:
\begin{enumerate}
\item
$\Pi\cong\mathbb{Z}[\pi]\oplus{L}$, where the first summand
is self-annihilating with respect to $\lambda_X$;
and 
\item{}the image of $k_1(X)$ generates the summand $H^3(\pi;\mathbb{Z}[\pi])$.
\end{enumerate}

The dual $L^\dagger$ is projective, 
and is of stable rank 1 if $\chi(X)=2$.
We show first that $L\cong{I_\pi}$.
Condition (1) implies that 
$ev:\Pi\to\Pi^\dagger\cong\mathbb{Z}[\pi]\oplus{L^\dagger}$ 
(the adjoint of $\lambda_X$)
has block-matrix form
\[
\left(\smallmatrix 0&\tilde\beta\\
\alpha&\tilde\gamma \endsmallmatrix\right),
\]
where $\alpha\in{L^\dagger}$,
$\tilde\beta:L\to\mathbb{Z}[\pi]$
and $\tilde\gamma:L\to{L^\dagger}$.
Considering the exact sequence of Lemma 2, 
we see that $\tilde\beta$ must be injective,
since $ev$ is injective, and $\mathrm{Cok}(\tilde\beta)=\mathbb{Z}$ or 0.
But if $\tilde\beta$ is an isomorphism then $\Pi\cong\mathbb{Z}[\pi]\oplus{J^\dagger}$, 
and so is projective.
This contradicts $[\Pi]_{pr}=[I_\pi]_{pr}$.
Hence $\mathrm{Cok}(\tilde\beta)=\mathbb{Z}$, 
and so $L\cong\mathrm{Im}(\tilde\beta)=I_\pi$.

Since $I_\pi^\dagger\cong\mathbb{Z}[\pi]$ and 
$Hom_{\mathbb{Z}[\pi]}(I_\pi,\mathbb{Z}[\pi])\cong\mathbb{Z}[\pi]j_\pi$,
we may now write $\tilde\beta=\beta.j_\pi$ and $\tilde\gamma=\gamma.j_\pi$,
where $\alpha,\beta,\gamma\in\mathbb{Z}[\pi]$.
We have $\alpha=\overline\beta$ and $\overline{\gamma}=\gamma$,
by the hermitean symmetry of $\lambda_X$.
Moreover $\beta$ is not a zero-divisor in $\mathbb{Z}[\pi]$,
since $ev$ is injective, 
and $\beta.j_\pi$ must map $I_\pi$ onto $I_\pi$, 
by the exactness of the sequence in Lemma 2 above.
Thus $\beta$ is a unit in $End_{\mathbb{Z}[\pi]}(I_\pi)\cong\mathbb{Z}[\pi]$.
(Hence so is $\alpha=\overline\beta$.)
By using the change of basis $(x,y)\mapsto(x,\beta^{-1}y)$, 
for $x\in\mathbb{Z}[\pi]$ and $y\in{I_\pi}$ we may arrange that $\beta=1$
(and $\gamma$ becomes $\widehat\gamma=\alpha^{-1}\gamma\beta^{-1}$).
If we then use the change of basis $(x,y)\mapsto(x+\delta.y,y)$
we may replace $\widehat\gamma$ by $\widehat\gamma+\delta+\overline\delta$.
We see easily that there are at most two isometry classes of such pairings,
depending on the parity of $\varepsilon(\gamma)$.
As the image of $\varepsilon(\gamma)$ in $\mathbb{Z}/2\mathbb{Z}$ 
is invariant under all changes of basis,  
the two remaining possibilities are distinct.
Condition (2) implies that $k_1(X)$ is essentially unique.
Thus there are at most two such homotopy types.

If $Y$ is a 3-manifold this parity detects the twist $\tau$ \cite{Pl86}.
Is this true more generally?
If $Y$ is a homology 3-sphere and $\chi(X)=2$ then $X$ is a homology 4-sphere, 
so $v_2(X)=0$. Thus this parity is not directly related to standard characteristic classes.


\newpage
\begin{thebibliography}{99}

\bibitem{Ar81} Artamanov, V. A. 
Projective nonfree modules over group rings of solvable groups,
Mat. Sbornik 116 (1981), 232--244.

\bibitem{Ba} Baues, H. J.  
\textit{Combinatorial Homotopy and 4-Dimensional Complexes},

Expositions in Math. 2,  W. De Gruyter, Berlin -- New York (1991).

\bibitem{BB} Baues, H. J.  and Bleile, B.  
Poincar\'e duality complexes in dimension four,

Alg. Geom. Top. 8 (2008),  2355--2389.

\bibitem{Do60} Dold, A.  Zur Homotopietheorie der Kettenkomplexe,

Math. Ann. 140 (1960), 278--298.

\bibitem{FJ87} Farrell, F. T. and Jones, L. E.  
Implications of the Geometrization Conjecture for the algebraic $K$-theory of 3-manifolds,
in \textit{Geometry and Topology, Athens Ga (1985)},
Lecture Notes in Pure and Applied Math. 105 (1987), 109--113.

\bibitem{HK88} Hambleton, I, and Kreck, M. 
On the classification of topological 4-manifolds with finite fundamental group, 
Math. Ann. 280 (1988),  85--104. 

Corrigendum {\it ibid} 372 (2018), 527--530.

\bibitem{FMGK} Hillman, J.A. 
\textit{Four-Manifolds, Geometries and Knots},

GT Monographs, vol. 5, 
Geometry and Topology Publications, 

Warwick (2002 -- revised 2007, 2022).

\bibitem{PDD3} Hillman, J.A.  
\textit{Poincar\'e Duality in Dimension 3},

Open Book Series vol.  3, MSP, Berkeley (2020).

\bibitem{Hi04} Hillman, J. A.  
$PD_4$-complexes with free fundamental group,

Hiroshima Math. J. 34 (2004),   295--306.

\bibitem{Hi20} Hillman, J.A.  
$PD_4$-complexes and 2-dimensional duality groups,

2019-20 MATRIX Annals (2020), 57--109.

\bibitem{KLPT} Kasprowski, D., Land,  M., Powell, M. A.  and Teichner, P.
Stable classification of 4-manifolds with 3-manifold fundamental groups,

J.  Topol. 10 (2017), 821-887.

\bibitem{KPR} Kasprowski, D.,  Powell, M. A.  and Ray, A. 
On the homotopy classification of 4-manifolds with infinite fundamental group,
preprint (2023).

\bibitem{Pl86} Plotnick, S.  P. 
Equivariant intersection forms, knots in $S^4$, 
and rotations in 2-spheres,
Trans.  Amer.  Math.  Soc.  296 (1986),  543--575.

\bibitem{Ra} Ranicki, A. A. 
\textit{Algebraic and Geometric Surgery},

Oxford Mathematical Monographs (2002).

\end{thebibliography}
\end{document}